\documentclass[a4paper, 12pt]{amsart}
\usepackage{amsmath, amssymb, amscd, enumerate}

\newtheorem{theorem}{Theorem}

\newtheorem{lemma}[theorem]{Lemma}
\newtheorem{definition}[theorem]{Definition}

\newcommand{\sm}{\left(\begin{smallmatrix}}
\newcommand{\esm}{\end{smallmatrix}\right)}
\theoremstyle{remark}
\newtheorem{remark}{Remark}

\makeatletter
\def\@eqnnum{(\thesection.\theequation)}
\makeatother

\newcommand{\SL}{\mathrm{SL}}

\newcommand{\Mp}{\mathrm{Mp}}

\begin{document}

\title[On harmonic weak Maass forms]
{On harmonic weak Maass forms of half integral weight}

\author{Bumkyu Cho}
\address{Department of Mathematics, Pohang University of Science and
Technology, San 31, Hyoja-dong, Nam-gu, Pohang-si, Gyeongsangbuk-do
790-784, Republic of Korea}

\email{bam@math.kaist.ac.kr}

\author{YoungJu Choie}

\address{Department of Mathematics, Pohang Mathematics Institute(PMI),  POSTECH, Pohang, Korea}

\email{yjc@postech.ac.kr}

\thanks{The first author was partially supported by BK21 at
POSTECH, the Tae-Joon Park POSTECH Postdoctoral Fellowship, and
NRF 2010-0008426. The second author was partially supported by NRF
2009008-3919 and NRF-2010-0029683}


\dedicatory{}

\begin{abstract}
Since Zweger \cite{Ze} found a connection between mock theta
functions and harmonic Maass forms this subject has been a vast
research interest recently. Motivated by Zweger's work harmonic
Maass-Jacobi forms were introduced in \cite{BR}, which include the
classical Jacobi forms. We show the isomorphisms among the space $H_{k +
\frac{1}{2}}^{+}(\Gamma_0(4m))$ of (scalar valued) harmonic weak Maass forms of half integral weight whose Fourier coefficients are
supported on suitable progressions, the space $H_{k + \frac{1}{2}, \bar{\rho}_L}$ of vector valued ones, and the space $\mathbb{\widehat{
J}}_{k+1,m}^{cusp}$ of certain harmonic
Maass-Jacobi forms of integral weight:
$$H_{k + \frac{1}{2}}^{+}(\Gamma_0(4m)) \simeq H_{k + \frac{1}{2}, \bar{\rho}_L} \simeq \mathbb{\widehat{ J}}_{k+1,m}^{cusp}$$
for $k$ odd and $m = 1$ or a prime. This
is an extension of the result developed by Eichler and Zagier
\cite{EZ}, which showed the isomorphisms among the Kohnen plus space $M_{k + \frac{1}{2}}^{+}(\Gamma_0(4m))$ of (scalar valued) modular forms of half integral weight, the space $M_{k + \frac{1}{2}, \bar{\rho}_L}$ of vector valued ones, and the space $J_{k+1,m}$ of Jacobi forms of integral weight:
$$M_{k + \frac{1}{2}}^{+}(\Gamma_0(4m)) \simeq M_{k +
\frac{1}{2}, \bar{\rho}_L} \simeq J_{k+1,m}.$$

\end{abstract}

\maketitle

\section{\bf{Introduction and statement of a result}}

Let $k$ be an integer, and $m$ a positive integer. We denote by
$M_{k + \frac{1}{2}}(\Gamma_0(4m))$  and $M_{k+
\frac{1}{2}}^!(\Gamma_0(4m))$ the space of holomorphic and weakly
holomorphic  modular forms, respectively, of weight $k +
\frac{1}{2}$ for $\Gamma_0(4m)$.
%$M_{k+\frac{1}{2}}(m)$ is the usual holomorphic modular form of
%weight $k+\frac{1}{2}$ on $\Gamma_0(4m).$
%satisfying the Kohnen's plus space condition, that is, $f(\tau)$
%has a Fourier expansion of the form
%\[ f(\tau) = \sum_{(-1)^k n \equiv 0, 1 \bmod 4} c_f(n) q^n. \]
Further we define a subspace $M_{k +
\frac{1}{2}}^{+}(\Gamma_0(4m))$ of $M_{k +
\frac{1}{2}}(\Gamma_0(4m))$ by
\[ M_{k + \frac{1}{2}}^{+}(\Gamma_0(4m)) := \{ f \in M_{k + \frac{1}{2}}(\Gamma_0(4m)) \, | \, c_f(n)
= 0 \mbox{ unless } (-1)^k n \equiv \Box \bmod 4m \}. \]

Let $L$ be the lattice $2m\mathbb Z$ equipped with the quadratic
form $Q(x) = x^2/4m$. Then its dual $L'$ equals $\mathbb Z$. Let
$\rho_L$ denote the Weil representation associated to the
discriminant form $(L'/L, Q)$, and $\bar{\rho}_L$ its dual
representation. We denote by $M_{k + \frac{1}{2}, \rho_L}$ the
space of $\mathbb C[L'/L]$-valued holomorphic modular forms
of weight $k + \frac{1}{2}$ and type $\rho_L$. Then Eichler and
Zagier \cite[Theorems 5.1, 5.4, and 5.6]{EZ} proved the following
isomorphisms: for $k$ odd and $m=1$ or a prime,
$$
 M_{k + \frac{1}{2}}^{ +}(\Gamma_0(4m))  \simeq M_{k +
\frac{1}{2}, \bar{\rho}_L} \simeq J_{k+1,m}.$$
Here $J_{k+1,m}$ is the space of Jacobi forms
of weight $k+1$ and index $m$ on the full modular group $\Gamma(1)
$.

Our result extends this isomorphism to the spaces of harmonic weak Maass forms
(Theorem \ref{Theorem - isomorphism for half integral weight case} and
Theorem \ref{jacobi}): if  $k$ is odd and $m = 1$ or a prime,

 $$ H_{k + \frac{1}{2}}^{ +}(\Gamma_0(4m))  \simeq H_{k +
\frac{1}{2}, \bar{\rho}_L} \simeq \mathbb{\widehat{
J}}_{k+1,m}^{cusp}.$$

% $$\left\{\begin{array}{cc}
% H_{k + \frac{1}{2}}^{ +}(\Gamma_0(4m))  \simeq H_{k +
%\frac{1}{2}, \bar{\rho}_L} \simeq \mathbb{\widehat{
%J}}_{k+1,m}^{}, \, \mbox{if  $k$ is odd}\\
%H_{k + \frac{1}{2}}^{ +}(\Gamma_0(4m))  \simeq H_{k + \frac{1}{2},
%{\rho}_L} ,\,  \mbox{if  $k$ is  even}\end{array}\right.$$

Here $H_{k + \frac{1}{2}}^{ +}(\Gamma_0(4m))$, $H_{k +
\frac{1}{2}, \bar{\rho}_L}$, and $ \mathbb{\widehat{
J}}_{k+1,m}^{cusp}$ are the spaces consisting of corresponding
harmonic ones (see Section 2).

\medskip

In order to state our main results more precisely,
we let $k$ be an integer and fix $m = 1$ or a prime.
We denote by $H_{k + \frac{1}{2}}(\Gamma_0(4m))$ the space of harmonic weak
Maass forms of weight $k + \frac{1}{2}$ for $\Gamma_0(4m)$ (see
Section 2.1). Then it is known (see, for instance,
\cite{BF}) that $f \in H_{k + \frac{1}{2}}(\Gamma_0(4m))$ has a
unique decomposition $f = f^+ + f^-$, where
\begin{eqnarray*}
f^+(\tau) & = & \sum_{n \gg -\infty} c_f^+(n) q^n, \\
f^-(\tau) & = & \sum_{n < 0} c_f^-(n) \Gamma \left( \frac{1}{2} - k, 4\pi
|n| y \right) q^n.
\end{eqnarray*}
Here $\Gamma(a, y) = \int_y^\infty e^{-t} t^{a - 1} dt$ denotes
the incomplete Gamma function. We define a subspace $H_{k +
\frac{1}{2}}^{+}(\Gamma_0(4m))$ of $H_{k +
\frac{1}{2}}(\Gamma_0(4m))$ by
\[ H_{k + \frac{1}{2}}^{+}(\Gamma_0(4m)) := \{ f \in
H_{k + \frac{1}{2}}(\Gamma_0(4m)) \, | \, c_f^\pm(n) = 0 \mbox{
unless } (-1)^k n \equiv \Box \bmod 4m \}. \]

Let $H_{k + \frac{1}{2}, \rho_L}$ denote the space of $\mathbb
C[L'/L]$-valued harmonic weak Maass forms of weight $k +
\frac{1}{2}$ and type $\rho_L$ (see Section 2.2). We denote the
standard basis elements of the group algebra $\mathbb C[L'/L]$ by
$\mathfrak e_\gamma$ for $\gamma \in L'/L$. Suppose that the
discriminant form $(L'/L, Q)$ is given by $(\mathbb Z/2m\mathbb Z,
Q)$, where $Q(\gamma) = \gamma^2/4m$ for $\gamma \in \mathbb
Z/2m\mathbb Z$ with the signature $(b^+,b^-).$  Then the level of
$L$ equals $4m$, and $b^+ - b^- \equiv 1 \bmod 8$. For example if
we take
\[ L = \{ X = \left( \begin{smallmatrix} b & -a/m \\ c & -b
\end{smallmatrix} \right) \in \mathrm{Mat}_2(\mathbb Q) \, | \, a,
b, c \in \mathbb Z \} \] with $Q(X) = -m\det(X)$, then $(b^+, b^-) =
(2, 1)$ and
\[ L' = \{ X = \left( \begin{smallmatrix} b/2m & -a/m \\ c & -b/2m
\end{smallmatrix} \right) \in \mathrm{Mat}_2(\mathbb Q) \, | \, a,
b, c \in \mathbb Z \}. \] For a given $f \in H_{k +
\frac{1}{2}}^{+}(\Gamma_0(4m))$ we define a $\mathbb
C[L'/L]$-valued function $F = \sum_{\gamma \in \mathbb Z / 2m
\mathbb Z} F_\gamma \mathfrak e_\gamma$ by
\begin{eqnarray*}
F_\gamma(\tau) := \frac{1}{s(\gamma)} \sum_{n \in \mathbb Z \atop
(-1)^k n \equiv \gamma^2 \bmod 4m} c_f(n, y/4m) q^{n/4m}.
\end{eqnarray*}
Here $c_f(n, y) := c_f^+(n) + c_f^-(n)\Gamma(\frac{1}{2} - k, 4 \pi
|n| y)$, and $s(\gamma) = 1$ if $\gamma \equiv 0, m \bmod 2m$, and
$2$ otherwise.

\begin{theorem}\label{Theorem - isomorphism for half integral weight case}
With the notation as above we have the following.

(1) If $k$ is even, then the map $f \mapsto F$ defines an
isomorphism of $H_{k + \frac{1}{2}}^{+}(\Gamma_0(4m))$ onto $H_{k
+ \frac{1}{2}, \rho_L}$.

(2) If $k$ is odd, then the map $f \mapsto F$ defines an
isomorphism of $H_{k + \frac{1}{2}}^{+}(\Gamma_0(4m))$ onto $H_{k
+ \frac{1}{2}, \bar{\rho}_L}$.
\end{theorem}

\begin{remark}
(1) For a given vector valued modular form $F = \sum_\gamma F_\gamma
\mathfrak e_\gamma$ the map $F \mapsto f$ will be the inverse
isomorphism where $f(\tau) : = \sum_\gamma F_\gamma(4m\tau)$.

(2) If we restrict the domain on $M_{k +
\frac{1}{2}}^{!+}(\Gamma_0(4m))$, we get an isomorphism onto
$M_{k + \frac{1}{2}, \rho_L}^!$ ($k$ even), and so on.

(3) This kind of result for weakly holomorphic
modular forms of integral weight is proved by Bruinier and
Bundschuh \cite{BB}. Let $p$ be an odd prime. We write
$M_k^!(\Gamma_0(p), (\frac{\cdot}{p}))$ for the space of weakly
holomorphic modular forms of weight $k$ for $\Gamma_0(p)$ with
Nebentypus $(\frac{\cdot}{p})$. For $\epsilon \in \{ \pm 1 \}$ we
define the subspace
\[ M_k^{!\epsilon}(\Gamma_0(p), \big(\frac{\cdot}{p}\big))
:= \{ f\in M_k^!(\Gamma_0(p), \big(\frac{\cdot}{p}\big)) \, | \,
c_f(n) = 0 \mbox{ if } \big(\frac{n}{p}\big) = - \epsilon \}. \]
Let $L$ be the lattice so that the discriminant group $L'/L$ is
isomorphic to $\mathbb Z / p\mathbb Z$. On $L'/L$ the quadratic
form is equivalent to $Q(x) = \alpha x^2/p$ for some $\alpha \in
\mathbb Z / p\mathbb Z - \{ 0 \}$. Put $\epsilon =
(\frac{\alpha}{p})$. Then Bruinier and Bundschuh \cite[Theorem
5]{BB} showed that $M_k^{!\epsilon}(\Gamma_0(p),
(\frac{\cdot}{p}))$ is isomorphic to $M_{k, \rho_L}^!$. For
integral weight case Bruinier and Bundschuh's argument can be
applied to the spaces of harmonic weak Maass forms (see
\cite[Theorem 1.2]{CC}). However,  for half integral weight case,  we need another argument because Eichler and Zagier's argument depends on the dimension formulas for the spaces of holomorphic modular forms (see the
proof of \cite[Theorem 5.6]{EZ}). It is essential that our proof of Theorem \ref{Theorem - isomorphism for half integral weight case} relies on some nontrivial  properties of the Weil representation.
\end{remark}

\medskip

Next we show that the spaces in Theorem \ref{Theorem - isomorphism for half integral weight case} (2) are isomorphic to the space of
harmonic Maass-Jacobi forms recently developed by Bringmann and
Richter \cite{BR}.\\

Let $L$ be the lattice $2m \mathbb Z$ equipped with the positive
definite quadratic form $Q(x) = x^2/4m$. Then the space $J_{k, m}$
of Jacobi forms of weight $k$ and index $m$ is isomorphic to the
space $M_{k - \frac{1}{2}, \bar{\rho}_L}$ of $\mathbb
C[L'/L]$-valued holomorphic modular forms of weight $k -
\frac{1}{2}$ and type $\bar{\rho}_L$ (see \cite[Theorem 5.1]{EZ}).
Recently, Bringmann and Richter \cite{BR} introduced harmonic
Maass-Jacobi forms, which include the classical Jacobi forms. Let
$\hat{\mathbb J}_{k, m}^{cusp}$ be the space of certain harmonic
Maass-Jacobi forms
\[ \phi(\tau, z) = \sum_{n, r \in \mathbb Z \atop D \ll \infty} c^+(n, r) q^n \zeta^r + \sum_{n, r \in \mathbb Z \atop D > 0} c^-(n, r) \Gamma \left( \frac{3}{2} - k, \frac{\pi D y}{m} \right) q^n \zeta^r \]
of weight $k$ and index $m$ (see Section 2.3). Here, $D = r^2 - 4nm$, $q = e^{2\pi i \tau}$, $\zeta = e^{2\pi i z}$. By the transformation property of harmonic Maass-Jacobi forms \cite[Definition 3]{BR}, one can deduce that if $r' \equiv r \bmod 2m$ and $D' = D$ with $D' := r'^2 - 4n'm$, then
\[ c^{\pm}(n', r') = c^{\pm}(n, r), \quad \Gamma \left( \frac{3}{2} - k, \frac{\pi D' y}{m} \right) = \Gamma \left( \frac{3}{2} - k, \frac{\pi D y}{m} \right). \]
Hence, we can decompose $\phi(\tau, z)$ by a linear combination of the theta functions as
\[ \phi(\tau, z) = \sum_{\mu \in \mathbb Z / 2m \mathbb Z} h_\mu (\tau) \theta_{m, \mu} (\tau, z), \]
where
\[ h_\mu (\tau) := \sum_{N \gg -\infty} c^+ \left( \frac{N + r^2}{4m}, r \right) q^{N/4m} + \sum_{N < 0} c^- \left( \frac{N + r^2}{4m}, r \right) \Gamma \left( \frac{3}{2} - k, -\frac{\pi N y}{m} \right) q^{N/4m} \]
with any $r \in \mathbb Z$, $r \equiv \mu \bmod 2m$, and
\[ \theta_{m, \mu} (\tau, z) := \sum_{r \in \mathbb Z \atop r \equiv \mu \bmod 2m} q^{r^2/4m} \zeta^r. \]
Using the same argument in \cite[Theorem 5.1]{EZ}, the $2m$-tuples
$(h_\mu)_{\mu \, (2m)}$ satisfies the desired transformation
formula for vector valued harmonic weak Maass forms. Now the
remaining thing is to check $\Delta_{k - \frac{1}{2}} h_\mu = 0$.
By definition, $\phi(\tau, z)$ vanishes under the action of the
Casimir element $C^{k, m}$ (see \cite[p. 2305]{BR}), and the
action of $C^{k, m}$ on functions in $\hat{\mathbb J}_{k, m}^{cusp}$
agrees with that of
\[ C^{k, m} = -2 \Delta_{k - \frac{1}{2}} + \frac{(\tau - \bar{\tau})^2}{4\pi i m}\partial_{\bar{\tau}zz} \]
(see \cite[Proof of Lemma 1]{BR}). Using the fact that $\theta_{m,\mu}(\tau,z)$ is in the heat kernel, that is,
\[ \left( \partial_\tau  - \frac{1}{8 \pi i m} \partial_{zz} \right) \left( \theta_{m, \mu} \right) = 0, \]
one can conclude by a direct computation that
\[ C^{k, m} (\phi) = 0 \Longrightarrow \Delta_{k - \frac{1}{2}} \left( h_\mu \right) = 0 \] for all $\mu \in \mathbb Z / 2m \mathbb Z$. In conclusion,
\[ \hat{\mathbb J}_{k, m}^{cusp} \simeq H_{k - \frac{1}{2}, \bar{\rho}_L}. \]
Hence, we get the following theorem:

\begin{theorem}\label{jacobi}
Let $k$ be even, and $m = 1$ or a prime. Then
\[ \hat{\mathbb J}_{k, m}^{cusp} \simeq H_{k - \frac{1}{2}}^{+}(\Gamma_0(4m)). \]
\end{theorem}

\begin{remark}
(1) In the case of $k$ odd, $M_{k - \frac{1}{2}}^{+}(\Gamma_0(4m))$ is
isomorphic to the space of skew holomorphic Jacobi forms of weight
 $k$ and index $m$. So one can guess that there must be
 a similar isomorphism as above when $k$ is odd. To do that,
 we need to introduce a new definition of Maass Jacobi forms which include skew holomorphic Jacobi forms.
 We hope this can be done by following and modifying the work of Bringmann and Richter \cite{BR}.

(2) In the argument of the proof, the growth condition doesn't matter. Namely, if we redefine $H_{k - \frac{1}{2}, \bar{\rho}_L}$ and $H_{k - \frac{1}{2}}^{+}(\Gamma_0(4m))$ so that they have at most linear exponential growth at cusps, then
\[ \hat{\mathbb J}_{k, m} \simeq H_{k - \frac{1}{2}, \bar{\rho}_L} \simeq H_{k - \frac{1}{2}}^{+}(\Gamma_0(4m)) \]
for $k$ even and $m = 1$ or a prime. Here, $\hat{\mathbb J}_{k, m}$ is the corresponding bigger space (see Section 2.3).
\end{remark}

\section{\bf{Preliminaries}}

\subsection{Scalar valued modular forms}

Let $\tau = x + iy \in \mathbb H$, the complex upper half plane,
with $x, y \in \mathbb R$. Let $k \in \frac{1}{2}\mathbb Z - \mathbb
Z$, and $m$ a positive integer. Put $\varepsilon_d :=
(\frac{-1}{d})^{\frac{1}{2}}$.

Recall that \textit{weakly holomorphic modular forms of weight $k$
for $\Gamma_0(4m)$} are holomorphic functions $f : \mathbb H
\rightarrow \mathbb C$ which satisfy:
\begin{enumerate}[(i)]
\item For all $(\begin{smallmatrix} a & b \\ c & d
\end{smallmatrix}) \in \Gamma_0(4m)$ we have
\[ f\Big(\frac{a\tau + b}{c\tau + d}\Big) = \Big( \frac{c}{d} \Big)
\varepsilon_d^{-2k} (c\tau + d)^k f(\tau); \]

\item $f$ has a Fourier expansion of the form
\[ f(\tau) = \sum_{n \in \mathbb Z \atop n \gg -\infty} c_f(n) q^n,
\] and analogous conditions are required at all cusps.
\end{enumerate}
%We write $M_k^!(\Gamma_0(4m))$ for the space of these weakly
%holomorphic modular forms.

A smooth function $f : \mathbb H \rightarrow \mathbb C$ is called
a \textit{harmonic weak Maass form of weight $k$ for
$\Gamma_0(4m)$} if it satisfies:
\begin{enumerate}[(i)]
\item For all $(\begin{smallmatrix} a & b \\ c & d
\end{smallmatrix}) \in \Gamma_0(4m)$ we have
\[ f\Big(\frac{a\tau + b}{c\tau + d}\Big) = \Big( \frac{c}{d} \Big)
\varepsilon_d^{-2k} (c\tau + d)^k f(\tau); \]

\item $\Delta_k f = 0$, where $\Delta_k$ is the weight
$k$ hyperbolic Laplace operator defined by
\[ \Delta_k := -y^2 \Big( \frac{\partial^2}{\partial x^2} +
\frac{\partial^2}{\partial y^2} \Big) + i k y \Big(
\frac{\partial}{\partial x} + i\frac{\partial}{\partial y} \Big);
\]

\item There is a Fourier polynomial $P_f(\tau) = \sum_{-\infty \ll
n \leq 0} c_f^+(n) q^n \in \mathbb C[q^{-1}]$ such that $f(\tau) =
P_f(\tau) + O(e^{-\varepsilon y})$ as $y \rightarrow \infty$ for
some $\varepsilon > 0$. Analogous conditions are required at all
cusps.
\end{enumerate}
We denote the space of these harmonic weak Maass forms by
$H_k(\Gamma_0(4m))$. This space can be denoted by
$H_k^+(\Gamma_0(4m))$ in the context of \cite{BF}, which is the inverse image of $S_{2 - k}(\Gamma_0(4m))$ under the certain differential operator $\xi_k$. We have
$M_k^!(\Gamma_0(4m)) \subset H_k(\Gamma_0(4m))$. The polynomial
$P_f \in \mathbb C[q^{-1}]$ is called the principal part of $f$ at
the corresponding cusps.

%In particular $f \in H_k(\Gamma_0(4m))$ has a unique decomposition
%$f = f^+ + f^-$, where
%\begin{eqnarray*}
%f^+(\tau) & = & \sum_{n \gg -\infty} c_f^+(n) q^n, \\
%f^-(\tau) & = & \sum_{n < 0} c_f^-(n) \Gamma(1 - k, 4\pi |n| y) q^n.
%\end{eqnarray*}

\subsection{Vector valued modular forms}

We write $\Mp_2(\mathbb R)$ for the metaplectic two-fold cover of
$\SL_2(\mathbb R)$. The elements are pairs $(M, \phi)$, where $M =
(\begin{smallmatrix} a & b \\ c & d
\end{smallmatrix}) \in \SL_2(\mathbb R)$ and $\phi : \mathbb H
\rightarrow \mathbb C$ is a holomorphic function with $\phi(\tau)^2
= c\tau + d$. The multiplication is defined by
\[ (M, \phi(\tau)) (M', \phi'(\tau)) = (MM', \phi(M'\tau)
\phi'(\tau)). \] For $M = (\begin{smallmatrix} a & b \\ c & d
\end{smallmatrix}) \in \SL_2(\mathbb R)$ we use the notation
$\tilde{M} := ((\begin{smallmatrix} a & b \\ c & d
\end{smallmatrix}), \sqrt{c\tau + d}) \in \Mp_2(\mathbb R)$. We
denote by $\Mp_2(\mathbb Z)$ the integral metaplectic group, that is
the inverse image of $\SL_2(\mathbb Z)$ under the covering map
$\Mp_2(\mathbb R) \rightarrow \SL_2(\mathbb R)$. It is well known
that $\Mp_2(\mathbb Z)$ is generated by $T := \left(
(\begin{smallmatrix} 1 & 1 \\ 0 & 1 \end{smallmatrix}), 1 \right)$
and $S := \left( (\begin{smallmatrix} 0 & -1 \\ 1 & 0
\end{smallmatrix}), \sqrt{\tau} \right)$.

Let $(V, Q)$ be a non-degenerate rational quadratic space of
signature $(b^+, b^-)$. Let $L \subset V$ be an even lattice with
dual $L'$. We denote the standard basis elements of the group
algebra $\mathbb C[L' / L]$ by $\mathfrak e_\gamma$ for $\gamma
\in L'/L$, and write $\langle \cdot, \cdot \rangle$ for the
standard scalar product, anti-linear in the second entry, such
that $\langle \mathfrak e_\gamma, \mathfrak e_{\gamma'} \rangle =
\delta_{\gamma, \gamma'}$. There is a unitary representation
$\rho_L$ of $\Mp_2(\mathbb Z)$ on $\mathbb C[L'/L]$, the so-called
Weil representation, which is defined by
\begin{eqnarray*}
\rho_L(T) (\mathfrak e_\gamma) & := & e(Q(\gamma)) \mathfrak
e_\gamma, \\
\rho_L(S) (\mathfrak e_\gamma) & := & \frac{e((b^- -
b^+)/8)}{\sqrt{|L'/L|}} \sum_{\delta \in L'/L} e(-(\gamma, \delta))
\mathfrak e_\delta,
\end{eqnarray*}
where $e(z) := e^{2 \pi i z}$ and $(X, Y) := Q(X + Y) - Q(X) - Q(Y)$
is the associated bilinear form. We denote by $\bar{\rho}_L$ the
dual representation of $\rho_L$.

Let $k \in \frac{1}{2}\mathbb Z$. A holomorphic function $f :
\mathbb H \rightarrow \mathbb C[L'/L]$ is called a \textit{weakly
holomorphic modular form of weight $k$ and type $\rho_L$ for the
group $\Mp_2(\mathbb Z)$} if it satisfies:
\begin{enumerate}[(i)]
\item $f(M\tau) = \phi(\tau)^{2k} \rho_L(M, \phi) f(\tau)$ for all
$(M, \phi) \in \Mp_2(\mathbb Z)$;

\item $f$ is meromorphic at the cusp $\infty$.
\end{enumerate}
Here condition (ii) means that $f$ has a Fourier expansion of the
form
\[ f(\tau) = \sum_{\gamma \in L'/L} \sum_{n \in \mathbb Z + Q(\gamma)
\atop n \gg -\infty} c_f(\gamma, n) e(n\tau) \mathfrak e_\gamma. \]
The space of these $\mathbb C[L'/L]$-valued weakly holomorphic
modular forms is denoted by $M_{k, \rho_L}^!$. Similarly we can
define the space $M_{k, \bar{\rho}_L}^!$ of $\mathbb C[L'/L]$-valued
weakly holomorphic modular forms of type $\bar{\rho}_L$.

A smooth function $f : \mathbb H \rightarrow \mathbb C[L'/L]$ is
called a \textit{harmonic weak Maass form of weight $k$ and type
$\rho_L$ for the group $\Mp_2(\mathbb Z)$} if it satisfies:
\begin{enumerate}[(i)]
\item $f(M\tau) = \phi(\tau)^{2k} \rho_L(M, \phi) f(\tau)$ for all
$(M, \phi) \in \Mp_2(\mathbb Z)$;

\item $\Delta_k f = 0$;

\item There is a Fourier polynomial $P_f(\tau) = \sum_{\gamma \in L'/L}
\sum_{n \in \mathbb Z + Q(\gamma) \atop -\infty \ll n \leq 0}
c_f^+(\gamma, n) e(n\tau) \mathfrak e_\gamma$ such that $f(\tau) =
P_f(\tau) + O(e^{-\varepsilon y})$ as $y \rightarrow \infty$ for
some $\varepsilon > 0$.
\end{enumerate}
We denote by $H_{k, \rho_L}$ the space of these $\mathbb
C[L'/L]$-valued harmonic weak Maass forms. This space is denoted by
$H_{k, L}^+$ in \cite{BF}, which is the inverse image of $S_{2 - k, L^{-}}$ under $\xi_k$. We have $M_{k, \rho_L}^! \subset H_{k, \rho_L}$. Similarly we define the space $H_{k, \bar{\rho}_L}$. In
particular $f \in H_{k, \rho_L}$ has a unique decomposition $f = f^+
+ f^-$, where

\begin{eqnarray*}
f^+(\tau) & = & \sum_{\gamma \in L'/L} \sum_{n \in \mathbb Z + Q(\gamma)
\atop n \gg -\infty} c_f^+(\gamma, n) e(n\tau) \mathfrak e_\gamma, \\
f^-(\tau) & = & \sum_{\gamma \in L'/L} \sum_{n \in \mathbb Z +
Q(\gamma) \atop n < 0} c_f^-(\gamma, n) \Gamma(1 - k, 4\pi |n| y)
e(n\tau) \mathfrak e_\gamma.
\end{eqnarray*}

\subsection{Harmonic Maass-Jacobi forms} The most part of this
session we follow the notation given in \cite{BR}.

\begin{definition} A function $\phi:\mathbb{H} \times \mathbb{C}
\rightarrow \mathbb{C}$ is a harmonic Maass-Jacobi form of weight
$k$ and index $m$ if $\phi$ is real-analytic in
$\tau\in\mathbb{H}$ and $z\in \mathbb{C}$, and satisfies the
following conditions:

\begin{enumerate}
\item For all $A=[\sm a & b\\c& d\esm, (\lambda, \mu)]\in
\mathrm{SL}_2(\mathbb{Z})\ltimes \mathbb{Z}^2$
$$\phi\left(\frac{a\tau+b}{c\tau+d},
\frac{z+\lambda\tau+\mu}{c\tau+d}\right)(c\tau+d)^{-k}e^{2\pi i
m(-\frac{c(z+\lambda\tau+\mu)^2}{c\tau+d}+\lambda^2\tau+2\lambda
z)}=\phi(\tau,z).$$ \item $ C^{k,m}(\phi)=0,$  where $C^{k,m}$ is
the Casimir element of the real Jacobi group (see p. 2305 in
\cite{BR}).

\item $\phi(\tau,z)=O (e^{ay}e^{2\pi m v^2/y} )$ as
$y\rightarrow \infty$ for some $a>0$.

\end{enumerate}

\end{definition}

Let $\widehat{\mathbb{J}}_{k,m}$ be the space of harmonic Maass-Jacobi forms of weight $k$ and index $m$, which are holomorphic in $z$. In fact, we are interested in the subspace $\widehat{{\mathbb{J}}}_{k,m}^{cusp}$ consisting of the elements $\phi \in \widehat{\mathbb{J}}_{k,m}$ whose Fourier expansion is of the form
\[ \phi(\tau, z) = \sum_{n, r \in \mathbb Z \atop D \ll \infty} c^+(n, r) q^n \zeta^r + \sum_{n, r \in \mathbb Z \atop D > 0} c^-(n, r) \Gamma \left( \frac{3}{2} - k, \frac{\pi D y}{m} \right) q^n \zeta^r. \]
The space $\widehat{{\mathbb{J}}}_{k,m}^{cusp}$ is in fact the inverse image of $J_{3 - k, m}^{sk, cusp}$ under the certain differential operator $\xi_{k, m}$ (see \cite[Remarks (1) on p. 2307]{BR}).

\section{\bf{Proof of Theorem \ref{Theorem - isomorphism for half integral weight
case}}}

We will only give a proof of (1) because exactly the same argument
can be applied. We first prove that for a given $f \in
H_{k+\frac{1}{2}}^{+}(\Gamma_0(4m))$ the $\mathbb C[L'/L]$-valued
function $F$ as defined in Section 1 belongs to $H_{k+\frac{1}{2},
\rho_L}$. One has that $f(\tau) = \sum_{\gamma\in L'/L}
F_\gamma(4m\tau)$ by inspecting the Fourier expansion of $f$.
Since it is straightforward to check $F(\tau + 1) = \rho_L(T)
F(\tau)$, we show that
\begin{eqnarray}\label{Equation - rho(S)}
F \Big( -\frac{1}{\tau} \Big) = \tau^{k+\frac{1}{2}} \rho_L(S)
F(\tau).
\end{eqnarray}

In \cite{Kim} Kim proved (\thesection.\ref{Equation - rho(S)}) for
$m \in \mathfrak S$, where
\[ \mathfrak S = \{ 2, 3, 5, 7, 11, 13, 17, 19, 23, 29, 31, 41, 47,
59, 71 \}. \] We will prove (\thesection.\ref{Equation - rho(S)})
for $m = 1$ or a prime by following his argument. For details we
refer to \cite[pp. 735-737]{Kim}. For $j$ prime to $4m$ put
\[ f_j := f |_{k+\frac{1}{2}} \big( \big( \begin{smallmatrix} 1 & j \\ 0 & 4m
\end{smallmatrix} \big), (4m)^{1/4} \big) |_{k+\frac{1}{2}} W_{4m}, \] where
\[ W_{4m} := \big( \big( \begin{smallmatrix} 0 & -1 \\ 4m & 0
\end{smallmatrix} \big), (4m)^{1/4} \sqrt{-i\tau} \big). \]

We choose $b, d \in \mathbb Z$ so that $jd - 4mb = 1$. Then one
finds that
\begin{eqnarray*}
\big( \big( \begin{smallmatrix} 1 & j \\ 0 & 4m \end{smallmatrix}
\big), (4m)^{1/4} \big) W_{4m} & = & \big( \big( \begin{smallmatrix}
4mj & -1 \\ 16m^2 & 0 \end{smallmatrix} \big), 2 \sqrt{-mi\tau}
\big) \\
& = & \big( M, J(M, \tau) \big) \big( \big( \begin{smallmatrix} 4m
& -d \\ 0 & 4m \end{smallmatrix} \big), \psi_j \big),
\end{eqnarray*}
where
\[ \psi_j := \Big( \frac{4m}{j} \Big) \sqrt{\Big( \frac{-1}{j} \Big)}
e(-1/8) \] and $M = \big( \begin{smallmatrix} j & b \\ 4m & d
\end{smallmatrix} \big)$ and $J(M, \tau)$ denotes the automorphy factor
for the theta function $\sum_{n \in \mathbb Z} q^{n^2}$, that is,
$$J(M,\tau):=\Big(\frac{c}{d}\Big)\varepsilon_d^{-1}(c\tau+d)^{\frac{1}{2}},
\quad M=\sm * & * \\ c & d \esm \in \Gamma_0(4m).$$ Here
$(\frac{\cdot}{\cdot})$ is the usual Jacobi symbol.

This implies that
\begin{eqnarray}\label{Equation - 1st description of fj}
f_j & = & f |_{k + \frac{1}{2}} \big( \big( \begin{smallmatrix} 1 &
j \\ 0 & 4m \end{smallmatrix} \big), (4m)^{1/4} \big) W_{4m} \nonumber \\
& = & \psi_j^{-2k - 1} f \Big( \tau - \frac{j^{-1}}{4m} \Big) \nonumber \\
& = & \psi_j^{-2k - 1} \sum_{\gamma (2m)} e \Big( -\frac{j^{-1}
\gamma^2}{4m} \Big) F_\gamma(4m\tau),
\end{eqnarray}
where $j^{-1}$ denotes an integer which is the inverse of $j$ in
$(\mathbb Z / 4m \mathbb Z)^\times$. On the other hand we have by
definition that
\begin{eqnarray}\label{Equation - 2nd description of fj}
f_j & = & (4m)^{(-2k - 1)/4} \Big( \sum_{\gamma (2m)} e \Big(
\frac{j\gamma^2}{4m} \Big) F_\gamma \Big) \Big|_{k + \frac{1}{2}} W_{4m} \nonumber \\
& = & (4m)^{(-2k - 1)/2} \sqrt{-i\tau}^{-2k - 1} \sum_{\gamma (2m)}
e \Big( \frac{j\gamma^2}{4m} \Big) F_\gamma \Big( -\frac{1}{4m\tau}
\Big).
\end{eqnarray}
Replacing $\tau$ by $\tau/4m$ in (\thesection.\ref{Equation - 1st
description of fj}) and (\thesection.\ref{Equation - 2nd
description of fj}) one has the following identity:
\begin{eqnarray}\label{Equation - identity of 1st and 2nd description of fj}
\sum_{\gamma (2m)} e \Big( \frac{j\gamma^2}{4m} \Big) F_\gamma \Big(
-\frac{1}{\tau} \Big) = \Big( \frac{4m}{j} \Big) \sqrt{\Big(
\frac{-1}{j} \Big)}^{-1} \tau^{k+\frac{1}{2}} \sum_{\gamma (2m)} e
\Big( -\frac{j^{-1}\gamma^2}{4m} \Big) F_\gamma(\tau).
\end{eqnarray}

Let $R$ be a $2m \times 2m$ matrix defined by
\[ R := \frac{e(-1/8)}{\sqrt{2m}} \left( e \Big( -\frac{l\gamma}{2m}
\Big) \right)_{l (2m), \, \gamma (2m)}. \] In order to show $F \in
H_{k+\frac{1}{2}, \rho_L}$ we first need to prove
(\thesection.\ref{Equation - rho(S)}), i.e.
\[ \begin{pmatrix} \vdots \\ F_\gamma \\ \vdots \end{pmatrix} (
-1/\tau ) = \tau^{k+\frac{1}{2}} R \begin{pmatrix} \vdots \\
F_\gamma \\ \vdots \end{pmatrix} (\tau). \] Since $F_\gamma =
F_{-\gamma}$ the above identity is equivalent to
\[ B \begin{pmatrix} \vdots \\ F_\gamma \\ \vdots \end{pmatrix} (
-1/\tau ) = \tau^{k+\frac{1}{2}} BR \begin{pmatrix} \vdots \\
F_\gamma \\ \vdots \end{pmatrix} (\tau) \] for some matrix $B$
with $2m$ columns of which the first $m + 1$ ones are linearly
independent. For instance, in \cite{Kim} $B$ was chosen as
\[ A := \left( e \Big( \frac{\j_{\ell}\gamma^2}{4m} \Big)
\right)_{\ell(\varphi(4m)), \, \gamma (2m)} \] and checked that its
rank is $m + 1$ if $m \in \mathfrak S - \{ 2 \}$. Here $\j_{\ell}$
is the $\ell$th largest element in $\{ j \, | \, 1\leq j\leq 4m,
(j,4m)=1\}.$

In this paper we take $B$ as $CA$ where
\[ C := \left( e \Big( -\frac{\j_{\ell}\beta^2}{4m} \Big) \right)_{\beta
(2m), \, \ell (\varphi(4m))}. \]

\begin{lemma}\label{Lemma - rank of B}
The rank of $B := CA$ is $2 \varphi(m)$, and the first $2
\varphi(m)$ columns of $B$ are linearly independent.
\end{lemma}

\begin{proof}
The $\beta\gamma$-th entry $b_{\beta\gamma}$ of the $2m \times 2m$
matrix $B$ is given by
\begin{eqnarray*}
b_{\beta\gamma} & = & \sum_{\ell (\varphi(4m))} e \Big( \frac{\j_\ell(\gamma^2 -
\beta^2)}{4m} \Big) \\
& = & \left\{
\begin{array}{ll}
2 \varphi(m) & \mbox{if } \beta = \gamma \\
-2 & \mbox{if } \beta \neq \gamma \mbox{ and } \beta \equiv
\gamma \bmod 2 \\
0 & \mbox{otherwise.}
\end{array} \right.
\end{eqnarray*}
From this one can easily infer that $B$ has rank $2 \varphi(m)$,
and its first $2 \varphi(m)$ columns are linearly independent.
\end{proof}

Since
\[ AR = \left( \Big( \frac{4m}{\j_\ell} \Big) \sqrt{\Big( \frac{-1}{\j_\ell}
\Big)}^{-1} e \Big( -\frac{\j_\ell^{-1}\gamma^2}{4m} \Big)
\right)_{\ell (\varphi(4m)), \, \gamma (2m)} \] from the Gauss sum
formula (see \cite[p. 736]{Kim}), the identity
(\thesection.\ref{Equation - identity of 1st and 2nd description of
fj}) is equivalent to
\[ A \begin{pmatrix} \vdots \\ F_\gamma \\ \vdots \end{pmatrix} (
-1/\tau ) = \tau^{k+\frac{1}{2}} AR \begin{pmatrix} \vdots \\
F_\gamma \\ \vdots \end{pmatrix} (\tau). \] This implies that the
identity
\[ B \begin{pmatrix} \vdots \\ F_\gamma \\ \vdots \end{pmatrix} (
-1/\tau ) = \tau^{k+\frac{1}{2}} BR \begin{pmatrix} \vdots \\
F_\gamma \\ \vdots \end{pmatrix} (\tau) \] holds true. Combining
with Lemma \ref{Lemma - rank of B}, $F$ satisfies the
transformation property (\thesection.\ref{Equation - rho(S)}) for
$m \neq 2$. From (\thesection.\ref{Equation - 1st description of
fj}), (\thesection.\ref{Equation - identity of 1st and 2nd
description of fj}), and Lemma \ref{Lemma - rank of B} we can
infer that each $F_\gamma|_{k + \frac{1}{2}}((\begin{smallmatrix}
0 & -1 \\ 1 & 0
\end{smallmatrix}), \sqrt{\tau})$ can be written as a linear combination of
$f_j := f |_{k+\frac{1}{2}} ( ( \begin{smallmatrix} 1 & j \\
0 & 4m \end{smallmatrix} ), (4m)^{1/4} ) |_{k+\frac{1}{2}} W_{4m}$
for all $(j, 4m) = 1$. Since $\Delta_{k + \frac{1}{2}}$ commutes
with the Petersson slash operator (see \cite{Ma}), each $\Delta_{k +
\frac{1}{2}} F_\gamma$ vanishes, i.e. $F \in H_{k + \frac{1}{2},
\rho_L}$ for $m \neq 2$. The same procedure as given in \cite[Remark
3.2]{Kim} can be applied to the case when $m=2$, so we omit the
detailed proof.

Now we consider the converse. For a given $F = \sum_{\gamma (2m)}
F_\gamma \mathfrak e_\gamma \in H_{k+\frac{1}{2}, \rho_L}$ we define
\begin{eqnarray}\label{Equation - inverse isomorphism}
f(\tau) & := & \sum_{\gamma (2m)} F_\gamma (4m\tau).
\end{eqnarray}
It is straightforward to verify that $f$ satisfies the condition
$c_f^\pm(n) = 0$ unless $(-1)^k n \equiv \Box \bmod 4m$ by
inspecting the Fourier expansion of $F$. Also $\Delta_{k +
\frac{1}{2}} f = 0$ by (\thesection.\ref{Equation - inverse
isomorphism}). Thus it suffices to show that
\begin{eqnarray*}
f \Big( \frac{a\tau + b}{c\tau + d} \Big) = \Big( \frac{c}{d} \Big)
\sqrt{ \Big( \frac{-1}{d} \Big) }^{-1} (c\tau + d)^{k+\frac{1}{2}}
f(\tau)
\end{eqnarray*}
for all $(\begin{smallmatrix} a & b \\ c & d \end{smallmatrix}) \in
\Gamma_0(4m)$. We may assume that $d > 0$ by multiplying
$(\begin{smallmatrix} -1 & 0 \\ 0 & -1 \end{smallmatrix})$ if
necessary.

\begin{remark}
In what follows, our results hold for arbitrary $m > 0$. In fact one
may apply our argument even for somewhat general discriminant forms
$(L'/L, Q)$.
\end{remark}

\begin{lemma}\label{Lemma - trivial action by applying Shintani's}
For any $n \in \mathbb Z$ one has
\[ \rho_L \widetilde{ \begin{pmatrix} 1 & 0 \\ n & 1
\end{pmatrix} } \sum_{\gamma \in L'/L} \mathfrak e_\gamma
= \sum_{\gamma \in L'/L} \mathfrak e_\gamma. \]
\end{lemma}

\begin{proof}
Since $\widetilde{ ( \begin{smallmatrix} 1 & 0 \\ n & 1
\end{smallmatrix} ) } = \widetilde{ ( \begin{smallmatrix} 1 & 0 \\ 1 & 1
\end{smallmatrix} ) }^n$ it suffices to prove that
\[ \rho_L \widetilde{ \begin{pmatrix} 1 & 0 \\ 1 & 1
\end{pmatrix} } \sum_{\gamma \in L'/L} \mathfrak e_\gamma
= \sum_{\gamma \in L'/L} \mathfrak e_\gamma. \] For $M \in
\SL_2(\mathbb Z)$ we define the coefficients $\rho_{\beta \gamma}
(\tilde{M})$ of the representation $\rho_L$ by $\rho_{\beta
\gamma}(\tilde{M}) = \langle \rho_L(\tilde{M}) \mathfrak e_\gamma,
\mathfrak e_\beta \rangle$. From Shintani's result
\cite[Proposition 1.1]{Br} for $\rho_{\beta \gamma}(\tilde{M})$
 one has
\begin{eqnarray*}
\sum_{\gamma \in L'/L} \rho_{\beta\gamma} \widetilde{
\begin{pmatrix} 1 & 0 \\ 1 & 1 \end{pmatrix} } & = &
\sum_{\gamma \in L'/L} \frac{e((b^- - b^+)/8)}{\sqrt{2m}} e(Q(\beta)
- (\beta, \gamma) + Q(\gamma)) \\
& = & \frac{e((b^- - b^+)/8)}{\sqrt{2m}} \sum_{\gamma \in L'/L}
e(Q(\beta - \gamma)) \\
& = & 1   .
\end{eqnarray*}
The last equality is from Milgram's formula (see \cite{Bo}), that
is,
\[\sum_{\gamma\in
L'/L}e(Q(\gamma))=\sqrt{|L'/L|}e((b^+ - b^-)/8). \]
\end{proof}

First notice that
\[ f(\tau) = \langle F(4m\tau), \sum_{\gamma\in L'/L} \mathfrak e_\gamma
\rangle. \] Since $( \begin{smallmatrix} 4m & 0 \\ 0 & 1
\end{smallmatrix} ) ( \begin{smallmatrix} a & b \\ c & d
\end{smallmatrix} ) = ( \begin{smallmatrix} a & 4mb \\ c/4m &
d \end{smallmatrix} ) ( \begin{smallmatrix} 4m & 0 \\ 0 & 1
\end{smallmatrix} )$ we get for $( \begin{smallmatrix} a & b \\ c &
d \end{smallmatrix} ) \in \Gamma_0(4m)$ that
\begin{eqnarray}\label{Equation - via the standard scalar product}
f \Big( \frac{a\tau + b}{c\tau + d} \Big) & = & \Big\langle F
\big( \big( \begin{smallmatrix} a & 4mb \\ c/4m & d
\end{smallmatrix} \big) (4m\tau) \big), \sum_{\gamma\in L'/L}
\mathfrak e_\gamma \Big\rangle
\nonumber \\
& = & \Big\langle (c\tau + d)^{k+\frac{1}{2}} \rho_L \widetilde{
\big(
\begin{smallmatrix} a & 4mb \\ c/4m & d \end{smallmatrix} \big) }
F(4m\tau), \sum_{\gamma\in L'/L} \mathfrak e_\gamma \Big\rangle \nonumber \\
& = & (c\tau + d)^{k+\frac{1}{2}} \Big\langle F(4m\tau), \rho_L
\widetilde{ \big( \begin{smallmatrix} a & 4mb \\ c/4m & d
\end{smallmatrix} \big) }^{-1} \sum_{\gamma\in L'/L} \mathfrak e_\gamma
\Big\rangle.
\end{eqnarray}
Observe that $\Gamma^0(4m) = \langle \Gamma_0(4m) \cap \Gamma^0(4m),
( \begin{smallmatrix} 1 & 0 \\ 1 & 1 \end{smallmatrix} ) \rangle$.
More precisely one has
\[ \big( \begin{smallmatrix} a & 4mb \\ c/4m & d \end{smallmatrix}
\big) = \big( \begin{smallmatrix} a(1 - bc) & 4mb \\ -4mb(c/4m)^2
& d \end{smallmatrix} \big) \big( \begin{smallmatrix} 1 & 0 \\
ac/4m & 1 \end{smallmatrix} \big). \] We first consider the case
$a > 0$. The consistency condition implies that
\[ \widetilde{ \big( \begin{smallmatrix} a & 4mb \\ c/4m & d
\end{smallmatrix} \big) } = \widetilde{ \big( \begin{smallmatrix}
a(1 - bc) & 4mb \\ -4mb(c/4m)^2 & d \end{smallmatrix} \big) }
\widetilde{ \big( \begin{smallmatrix} 1 & 0 \\ ac/4m & 1
\end{smallmatrix} \big) }. \] Since $\rho_L \widetilde{ (
\begin{smallmatrix} 1 & 0 \\ ac/4m & 1 \end{smallmatrix} ) }
\sum_{\gamma\in L'/L} \mathfrak e_\gamma = \sum_{\gamma\in L'/L}
\mathfrak e_\gamma$ by Lemma \ref{Lemma - trivial action by
applying Shintani's}, one finds from Borcherds' result
\cite[Theorem 5.4]{Bo} that
\begin{eqnarray*}
\rho_L \widetilde{ \big( \begin{smallmatrix} a & 4mb \\ c/4m & d
\end{smallmatrix} \big) } \sum_{\gamma\in L'/L} \mathfrak e_\gamma & = &
\rho_L \widetilde{ \big( \begin{smallmatrix} a(1 - bc) & 4mb \\
-4mb(c/4m)^2 & d \end{smallmatrix} \big) } \sum_{\gamma\in L'/L}
\mathfrak
e_\gamma \\
& = & \Big( \Big( \frac{-4mb}{d} \Big) \sqrt{\Big( \frac{-1}{d}
\Big) }^{-1} \Big)^{b^+ - b^- + (\frac{-1}{2m}) - 1} \Big(
\frac{d}{4m} \Big) \sum_{\gamma\in L'/L} \mathfrak e_\gamma \\
& = & \Big( \frac{c}{d} \Big) \sqrt{ \Big( \frac{-1}{d} \Big)
}^{-(\frac{-1}{m})} \Big( \frac{m}{d} \Big) \Big( \frac{d}{m}
\Big) \sum_{\gamma\in L'/L} \mathfrak e_\gamma.
\end{eqnarray*}
Because $\sqrt{ ( \frac{-1}{d} ) }^{1 - (\frac{-1}{m})} (
\frac{m}{d} ) ( \frac{d}{m} ) = 1$ we get the following identity:
\begin{eqnarray}\label{Equation - by applying Borcherds' theorem}
\rho_L \widetilde{ \big( \begin{smallmatrix} a & 4mb \\ c/4m & d
\end{smallmatrix} \big) } \sum_{\gamma\in L'/L} \mathfrak e_\gamma = \Big(
\frac{c}{d} \Big) \sqrt{ \Big( \frac{-1}{d} \Big) }^{-1}
\sum_{\gamma\in L'/L} \mathfrak e_\gamma.
\end{eqnarray}
We claim that (\thesection.\ref{Equation - by applying Borcherds'
theorem}) holds true for the case $a < 0$. If $c = 0$, then $a = d =
-1$ and thereby it is straightforward to verify
(\thesection.\ref{Equation - by applying Borcherds' theorem}). So we
assume that $c \neq 0$. If we choose $x \in \mathbb Z$ so that $a +
xc > 0$, then from the elementary identity
\[ \big( \begin{smallmatrix} 1 & 4mx \\ 0 & 1 \end{smallmatrix} \big)
\big( \begin{smallmatrix} a & 4mb \\ c/4m & d \end{smallmatrix}
\big) = \big( \begin{smallmatrix} a + xc & 4m(b + xd) \\ c/4m & d
\end{smallmatrix} \big) \]
one can see that (\thesection.\ref{Equation - by applying Borcherds'
theorem}) holds true even for $a < 0$. Now if we insert
(\thesection.\ref{Equation - by applying Borcherds' theorem}) into
(\thesection.\ref{Equation - via the standard scalar product}), we
obtain
\begin{eqnarray*}
f \Big( \frac{a\tau + b}{c\tau + d} \Big) & = & \Big( \frac{c}{d}
\Big) \sqrt{ \Big( \frac{-1}{d} \Big) }^{-1} (c\tau +
d)^{k+\frac{1}{2}}
\langle F(4m\tau), \sum_{\gamma\in L'/L} \mathfrak e_\gamma \rangle \\
& = & \Big( \frac{c}{d} \Big) \sqrt{ \Big( \frac{-1}{d} \Big) }^{-1}
(c\tau + d)^{k+\frac{1}{2}} f(\tau).
\end{eqnarray*}
This completes the proof of Theorem \ref{Theorem - isomorphism for
half integral weight case}.

\bibliographystyle{amsplain}

\end{document}